\documentclass[reqno]{amsart}
\title[Weighted partitions and a master identity]{Weighted partitions with interval restrictions: exact formulas and a bivariate master identity}
\usepackage{amssymb,amsmath,amsthm,epsfig,graphics,latexsym}
\usepackage[hidelinks]{hyperref}
\theoremstyle{definition}

 \usepackage{tikz}
\usetikzlibrary{arrows.meta,positioning,shapes.geometric}
\theoremstyle{plain}
\newtheorem{theorem}    {Theorem}

\newtheorem*{maintheorem} {Main Theorem}
\newtheorem{lemma}      {Lemma}
\newtheorem{corollary}  {Corollary}
\newtheorem{conjecture} {Conjecture}
\theoremstyle{remark}

\newtheorem{remark}{{\bf Remark}}
\numberwithin{equation}{section}
 \newcommand{\Mod}[1]{\ (\mathrm{mod}\ #1)}

\newcommand{\fr}{\frac}

\newmuskip\pFqskip
\mathchardef\pFcomma=\mathcode`, 

 \newmuskip\pGqskip
\mathchardef\pGcomma=\mathcode`, 

\begin{document}
\author[G. E. Andrews]{George E. Andrews}
\address{The Pennsylvania State University, University Park, Pennsylvania 16802, USA}
\email{andrews@math.psu.edu}

\author[M. El Bachraoui]{Mohamed El Bachraoui}
\address{Department of Mathematical Sciences,
United Arab Emirates University, PO Box 15551, Al-Ain, UAE}
\email{melbachraoui@uaeu.ac.ae}

\author[A. Dhar]{Aritram Dhar}
\address{Department of Mathematics, University of Florida, Gainesville, FL 32611, USA}
\email{aritramdhar@ufl.edu}

\author[A. Goswami]{Ankush Goswami}
\address{Department of Mathematics, University of Texas Rio Grande Valley, Edinburg, TX 78541, USA}
\email{ankush.goswami@utrgv.edu}

\author[R. Li]{Runqiao Li}
\address{Department of Mathematics, University of Texas Rio Grande Valley, Edinburg, TX 78541, USA}
\email{runqiao.li@utrgv.edu}
\begin{abstract}
Let $a_2''(n)$ and $b_2''(n)$ be the signed partition functions introduced by Andrews and El Bachraoui for interval-restricted partitions whose parts greater than $1$ are controlled by the smallest even part and by the number of ones.  We prove two conjectures for these functions.  The first gives the generating function for $a_2''(n)$ as an elementary rational term plus a false theta series with periodic signs; the second asserts that the companion coefficients $b_2''(n)$ take only the values $-1,0,1,2$.  The central structural result introduces an auxiliary variable $z$ recording the number of non-compulsory parts greater than $1$.  We obtain closed forms for the two resulting generating functions and prove the master identity $(1+q^2)\mathcal B(z,q)-(1+q)\mathcal A(z,q)=-q^4/(1-q^3)$ using both analytic and combinatorial techniques.  At $z=-1$, this identity, together with a Rogers--Fine evaluation, gives the false theta formula for $a_2''(n)$ and an explicit generating function for $b_2''(n)$.  The latter formula implies the asserted coefficient range and leads to an exact coefficient description of $b_2''(n)$.  We also include a direct Heine--Rogers--Fine proof of the false theta formula, ordinary and fixed-refinement consequences of the master identity, and the resulting quantum modular interpretation.
\end{abstract}
\date{}
\thanks{The first author is partially supported by Simons Foundation Grant 633284.}
\maketitle
\section{Introduction}
Partitions in which the parts are constrained to lie in a short interval often lead to generating functions with a rich basic hypergeometric structure. In~\cite{Andrews-Bachraoui conjectures}, the first two authors studied several such families and introduced signed refinements obtained by weighting each partition by $(-1)^{j+1}$, where $j$ is the number of parts greater than $1$. For two of these signed partition functions, denoted $a_2''(n)$ and $b_2''(n)$, computations suggested unexpectedly rigid behavior: the first was conjectured to have a triangular false theta description, while the second appeared to be confined to the set $\{-1,0,1,2\}$. The purpose of the present paper is to prove these conjectures by combining two complementary viewpoints.  The main structural viewpoint is that both conjectures are consequences of a bivariate identity involving an additional variable $z$.  This is the natural level of generality: the weighted identities arise from the specialization $z=-1$, while the same identity also controls the ordinary and fixed-refinement generating functions.  We also give a direct Heine--Rogers--Fine derivation of the false theta formula for $a_2''(n)$, which makes the false theta and quantum modular structure especially transparent.

We use the standard notation~\cite{Andrews, Gasper-Rahman}
\[
(z;q)_0 = 1,\quad  (z;q)_n = \prod_{j=0}^{n-1} (1-zq^j),\quad
(z;q)_{\infty} = \prod_{j=0}^{\infty} (1-zq^j).
\]
All generating-function identities below are interpreted as identities of formal power series in $q$.

We begin with the family $a_2(n)$ from~\cite[Definition 2]{Andrews-Bachraoui conjectures}. It counts partitions of $n$ such that there are $l-1$ ones, the smallest even part is $2k$, every part greater than $1$ lies in the interval $[2k,2l-2]$, and the odd parts greater than $1$ are distinct. If $j$ denotes the number of parts greater than $1$, then $a_2''(n)$ is the corresponding signed count in which each partition is counted with weight $(-1)^{j+1}$. The generating functions of the sequences $a_2(n)$ and $a_2''(n)$ are
\begin{align*}
\sum_{n\geq 1} a_2(n) q^n &= \sum_{\substack{l\geq 2\\1\leq k\leq l-1}}\fr{(-q^{2k+1};q^2)_{l-k-1}}{(q^{2k};q^2)_{l-k}}q^{2k+l-1} \\
&=q^3+q^4+2q^5+3q^6+5q^7+6q^8+\cdots, \\
\sum_{n\geq 0} a_2''(n) q^n &= \sum_{\substack{l\geq 2\\1\leq k\leq l-1}}\fr{(q^{2k+1};q^2)_{l-k-1}}{(-q^{2k};q^2)_{l-k}}q^{2k+l-1} \\
&=q^3+q^4+q^6+q^7+q^9+q^{12}+\cdots.
\end{align*}
The following conjecture regarding $a_2''(n)$ was stated in the earlier paper~\cite[Conjecture 1]{Andrews-Bachraoui conjectures}.
\begin{conjecture}\label{conj1}
There holds
\begin{equation}\label{conj1-id}
\sum_{n\geq 0} a_2''(n) q^n
=\fr{q}{1-q^2}+\fr{1}{1+q}\sum_{n\geq 0} \chi(n) q^{\binom{n+1}{2}+1},
\end{equation}
where
\[
\chi(n)
=\begin{cases}
-1 &\ \text{if\ }n\equiv 0,1\Mod{4}, \\
1 &\ \text{if\ }n\equiv 2,3\Mod{4}.
\end{cases}
\]
\end{conjecture}

We now turn to the companion family $b_2(n)$, where the interval of allowed parts is enlarged from $[2k,2l-2]$ to $[2k,2l-1]$. This apparently small change produces a more complicated signed generating function, but it again leads to a strikingly simple coefficient pattern.

The companion function $b_2(n)$ is defined analogously: there are $l-1$ ones, the smallest even part is $2k$, every part greater than $1$ lies in the enlarged interval $[2k,2l-1]$, and the odd parts greater than $1$ are distinct. We let $b_2''(n)$ be the corresponding signed count, again with weight $(-1)^{j+1}$ where $j$ is the number of parts greater than $1$. Then
\begin{align*}
\sum_{n\geq 1} b_2(n) q^n &= \sum_{\substack{l\geq 2\\1\leq k\leq l-1}}\fr{(-q^{2k+1};q^2)_{l-k}}{(q^{2k};q^2)_{l-k}}q^{2k+l-1}\\
&=q^3 + q^4 + 2 q^5 + 4 q^6 + 5 q^7 + 7 q^8 + 10 q^9 +\cdots, \\
\sum_{n\geq 0} b_2''(n) q^n &= \sum_{\substack{l\geq 2 \\1\leq k \leq l-1}}\fr{(q^{2k+1};q^2)_{l-k}}{(-q^{2k};q^2)_{l-k}}q^{2k+l-1} \\
&= q^3 + q^4 + q^7 + q^8 - q^{10} + 2 q^{12} - q^{14} + q^{15} \cdots.
\end{align*}
The original second conjecture was the following range assertion~\cite[Conjecture 2]{Andrews-Bachraoui conjectures}.
\begin{conjecture}\label{conj2}
For any positive integer $n$, we have $b_2''(n)\in\{-1,0,1,2\}$.
\end{conjecture}

The central point is that the two families should first be treated together. For a partition counted by $a_2(n)$ or $b_2(n)$, let $r$ denote the number of parts greater than $1$ other than the compulsory smallest even part $2k$. We write $a_2(n;r)$ and $b_2(n;r)$ for the corresponding refined counts and define
\[
\mathcal A(z,q):=\sum_{n\ge 0}\sum_{r\ge 0} a_2(n;r)z^r q^n,\qquad
\mathcal B(z,q):=\sum_{n\ge 0}\sum_{r\ge 0} b_2(n;r)z^r q^n.
\]
Then
\begin{align}
\mathcal A(z,q)
&=\sum_{\substack{l\ge 2\\1\le k\le l-1}}
\frac{(-zq^{2k+1};q^2)_{l-k-1}}{(zq^{2k};q^2)_{l-k}}q^{2k+l-1}, \label{eq zA-def}\\
\mathcal B(z,q)
&=\sum_{\substack{l\ge 2\\1\le k\le l-1}}
\frac{(-zq^{2k+1};q^2)_{l-k}}{(zq^{2k};q^2)_{l-k}}q^{2k+l-1}. \label{eq zB-def}
\end{align}
By construction,
\[
\mathcal A(1,q)=\sum_{n\ge 0} a_2(n)q^n,\qquad \mathcal B(1,q)=\sum_{n\ge 0} b_2(n)q^n.
\]
Moreover, if $j$ is the number of parts greater than $1$, then $j=r+1$; hence $(-1)^r=(-1)^{j+1}$. Therefore
\[
\mathcal A(-1,q)=\sum_{n\ge 0} a_2''(n)q^n,\qquad
\mathcal B(-1,q)=\sum_{n\ge 0} b_2''(n)q^n.
\]

Let
\[
F(z,q):=\sum_{n\ge 0}\frac{(-zq;q^2)_n q^n}{(zq^2;q^2)_n},
\qquad
H(z,q):=\sum_{n\ge 0}\frac{(-zq;q^2)_n q^n}{(z;q^2)_n}.
\]
Only the product $(1-z)H(z,q)$ occurs in the closed form for $\mathcal B(z,q)$; equivalently,
\[
(1-z)H(z,q)=\sum_{n\ge0}(1-zq^{2n})\frac{(-zq;q^2)_nq^n}{(zq^2;q^2)_n},
\]
so the expression is regular at $z=1$ in the combinations used below.
The following theorem gives the bivariate closed forms and the master identity relating them.

\begin{maintheorem}
One has
\begin{align}
z\mathcal A(z,q)&=\frac{q}{1+q}F(z,q)-\frac{q}{1-q^2}, \label{eq zA-closed}\\
z\mathcal B(z,q)&=\frac{1-z}{1+q}H(z,q)-\frac{1}{1-q^2}
+\frac{z}{(1+q)(1-q^3)}-\frac{zq^2}{1-q^3}, \label{eq zB-closed}
\end{align}
and consequently
\begin{equation}\label{eq master bivariate}
(1+q^2)\mathcal B(z,q)-(1+q)\mathcal A(z,q)=-\frac{q^4}{1-q^3}.
\end{equation}
\end{maintheorem}

Writing
\[
A(q):=\sum_{n\ge 0} a_2''(n)q^n,\qquad B(q):=\sum_{n\ge 0} b_2''(n)q^n,
\]
the specialization $z=-1$ of~\eqref{eq master bivariate} gives
\begin{equation}\label{eq weighted relation}
(1+q^2)B(q)-(1+q)A(q)=-\frac{q^4}{1-q^3}.
\end{equation}
Equivalently, for every integer $n\ge 0$,
\begin{equation}\label{eq weighted relation coeff}
b_2''(n)+b_2''(n-2)=a_2''(n)+a_2''(n-1)-\mathbf 1_{\{n\ge4,\ n\equiv1\pmod 3\}},
\end{equation}
where $a_2''(m)=b_2''(m)=0$ for $m<0$, and $\mathbf 1_{\mathcal P}$ denotes the indicator of the property $\mathcal P$.

The first main consequence is the proof of Conjecture~\ref{conj1}.
\begin{theorem}\label{thm main-1}
Conjecture~\ref{conj1} is true; that is,
\[
A(q)=\fr{q}{1-q^2}+\fr{1}{1+q}\sum_{n\geq 0} \chi(n) q^{\binom{n+1}{2}+1}.
\]
\end{theorem}
As a consequence, $a_2''(n)$ is completely determined by a simple parity rule on triangular intervals, recovering~\cite[Proposition 1]{Andrews-Bachraoui conjectures}.
\begin{corollary}\label{cor a2-0,1}
Set $\beta(1)=\beta(2)=0$.  For $n\geq 3$, define $\beta(n)$ as follows:

(a)\ If $\binom{2k}{2} < n\leq \binom{2k+1}{2}$ for some nonnegative integer $k$, then
\[
\beta(n)
=\begin{cases}
1 &\ \text{if $n$ is odd}, \\
0 &\ \text{if $n$ is even}.
\end{cases}
\]
(b)\ If $\binom{2k+1}{2} < n\leq \binom{2k+2}{2}$ for some nonnegative integer $k$, then
\[
\beta(n)
=\begin{cases}
1 &\ \text{if $n$ is even}, \\
0 &\ \text{if $n$ is odd}.
\end{cases}
\]
Then for any positive integer $n$, we have $a_2''(n)=\beta(n)$.
\end{corollary}

The second main consequence is the corresponding formula for $b_2''(n)$; the conjectured range follows from this formula.
\begin{theorem}\label{thm main-2}
There holds
\begin{equation}\label{conj2-id}
B(q)
=\fr{q+q^2+q^3-q^4}{(1-q)(1+q^2)(1+q+q^2)}+\fr{q}{1+q^2}\sum_{n\geq 0} \chi(n) q^{\binom{n+1}{2}}.
\end{equation}
Consequently, for any positive integer $n$, we have
\[
b_2''(n)\in\{-1,0,1,2\}.
\]
\end{theorem}
Thus Theorem~\ref{thm main-2} proves Conjecture~\ref{conj2}, and in fact gives a sharper generating-function statement than the original range assertion.

Two further consequences come directly from specializations of the bivariate identity.
Putting $z=1$ in~\eqref{eq master bivariate} immediately gives the ordinary specialization.
\begin{corollary}\label{cor ordinary relation}
There holds
\begin{equation}\label{eq ordinary relation}
(1+q^2)\sum_{n\ge 0} b_2(n)q^n-(1+q)\sum_{n\ge 0} a_2(n)q^n=-\frac{q^4}{1-q^3}.
\end{equation}
\end{corollary}

The fixed-refinement specialization gives a coefficient-wise relation for each positive power of $z$; its proof is given in Section~\ref{sec proofs corollaries}.
\begin{corollary}\label{cor fixed-r}
Let
\[
A_r(q):=\sum_{n\ge 0} a_2(n;r)q^n,\qquad B_r(q):=\sum_{n\ge 0} b_2(n;r)q^n.
\]
Then
\[
A_0(q)=B_0(q)=\frac{q^3}{(1-q)(1-q^3)},
\]
and for every integer $r\ge 1$,
\begin{equation}\label{eq fixed-r relation}
(1+q^2)B_r(q)=(1+q)A_r(q).
\end{equation}
Equivalently, for every $r\ge 1$ and every integer $n\ge 0$,
\begin{equation}\label{eq fixed-r coeff relation}
b_2(n;r)+b_2(n-2;r)=a_2(n;r)+a_2(n-1;r),
\end{equation}
where $a_2(m;r)=b_2(m;r)=0$ for $m<0$.
\end{corollary}

The range assertion in Theorem~\ref{thm main-2} can be sharpened to the following exact coefficient formula. Its proof is also given in Section~\ref{sec proofs corollaries}.

\begin{corollary}\label{cor exact formula for bpp}
Let $\rho(n)$ be the $12$-periodic function defined by
\[
\rho(n)=
\begin{cases}
1, & n\equiv 1,2,5,10\pmod{12},\\
2, & n\equiv 6,9\pmod{12},\\
-1, & n\equiv 4,7,8,11\pmod{12},\\
0, & n\equiv 0,3\pmod{12}.
\end{cases}
\]
For $1\le u\le 8$, set
\begin{equation}\label{eq mu-definition}
\mu_u(n)=
\begin{cases}
\dfrac{1-(-1)^n}{2}, & u=1,7,\\[6pt]
1, & u=2,6,\\[4pt]
1+\dfrac{1+(-1)^n}{2}, & u=3,5,\\[6pt]
2, & u=4,\\[4pt]
0, & u=8.
\end{cases}
\end{equation}
Let $T_s=\binom{s+1}{2}$ for $s\ge 0$.  Then $b_2''(0)=0$, and for every $n\ge 1$ with
\[
T_{8m+u-1}<n\le T_{8m+u}\qquad (m\ge 0,\ 1\le u\le 8),
\]
we have
\begin{equation}\label{eq exact bpp}
b_2''(n)=\rho(n)+\mu_u(n)\,\chi(n-1).
\end{equation}
\end{corollary}

The paper is organized as follows. In Section~\ref{sec proof master} we give an analytic proof of the Main Theorem. In Section ~\ref{combinatorialproofs}, we supply combinatorial proofs of the various identities appearing in the Main Theorem. Section~\ref{sec proofs main theorems} contains the proofs of the two main consequences: first the proof of Theorem~\ref{thm main-1} from the bivariate identity, then an independent Heine--Rogers--Fine proof of the same false theta formula, and finally the proof of Theorem~\ref{thm main-2}. Section~\ref{sec proofs corollaries} proves Corollaries~\ref{cor fixed-r} and~\ref{cor exact formula for bpp}. We end in Section~\ref{sec quantum} with the quantum modular consequence of the false theta formula.
\section{Proof of the Main Theorem}\label{sec proof master}
We shall use the identity of Andrews, Subbarao, and Vidyasagar~\cite{And-Sub-Vid}
\begin{equation}\label{Andr-Subb-Vidy}
\sum_{n\geq 0} \fr{(x;q)_n}{(y;q)_n} q^n
=\fr{q(x;q)_\infty}{y(y;q)_\infty \big(1-\fr{xq}{y}\big)} + \fr{1-\fr{q}{y}}{1-\fr{xq}{y}}.
\end{equation}

\begin{remark}\label{rem DG specialization}
The identity~\eqref{Andr-Subb-Vidy} is also the $\gamma=0$ specialization of the following three-parameter identity of Dixit and Goswami~\cite[Theorem~1.3]{Dixit-Goswami}:
\begin{align}
\sum_{n\ge 0}\frac{(\alpha,\gamma;q)_n}{(\beta,\alpha\gamma q^2/\beta;q)_n}q^n
&=\frac{\beta^{-1}q}{(1-\alpha q/\beta)(1-\gamma q/\beta)}
\frac{(\alpha,\gamma;q)_\infty}{(\beta,\alpha\gamma q^2/\beta;q)_\infty} \notag\\
&\quad+\frac{(1-q/\beta)(1-\alpha\gamma q/\beta)}{(1-\gamma q/\beta)(1-\alpha q/\beta)}. \label{eq DG three parameter}
\end{align}
Indeed, taking $\alpha=x$, $\beta=y$, and $\gamma=0$ gives~\eqref{Andr-Subb-Vidy}.  This observation suggests a natural route toward parameter-dependent variants of the bivariate identities below.  We do not pursue those extensions here.
\end{remark}

\begin{proof}[Proof of the Main Theorem]
We first record two simple shift formulas. For every integer $k\ge 1$, a direct index shift gives
\begin{align}
\sum_{n\ge 0}\frac{(-zq^{2k+1};q^2)_n q^n}{(zq^{2k+2};q^2)_n}
&=\frac{1-zq^{2k}}{q(1+zq^{2k-1})}
\sum_{n\ge 0}\frac{(-zq^{2k-1};q^2)_n q^n}{(zq^{2k};q^2)_n}
-\frac{1-zq^{2k}}{q(1+zq^{2k-1})}, \label{eq zshift1}\\
\sum_{n\ge 0}\frac{(-zq^{2k+1};q^2)_n q^n}{(zq^{2k};q^2)_n}
&=\frac{1-zq^{2k-2}}{q(1+zq^{2k-1})}
\sum_{n\ge 0}\frac{(-zq^{2k-1};q^2)_n q^n}{(zq^{2k-2};q^2)_n}
-\frac{1-zq^{2k-2}}{q(1+zq^{2k-1})}. \label{eq zshift2}
\end{align}
Indeed, for~\eqref{eq zshift1},
\begin{align*}
q\frac{1+zq^{2k-1}}{1-zq^{2k}}
\sum_{n\ge 0}\frac{(-zq^{2k+1};q^2)_n q^n}{(zq^{2k+2};q^2)_n}
&=\sum_{n\ge 0}\frac{(-zq^{2k-1};q^2)_{n+1} q^{n+1}}{(zq^{2k};q^2)_{n+1}}\\
&=\sum_{n\ge 1}\frac{(-zq^{2k-1};q^2)_n q^n}{(zq^{2k};q^2)_n}\\
&=\sum_{n\ge 0}\frac{(-zq^{2k-1};q^2)_n q^n}{(zq^{2k};q^2)_n}-1,
\end{align*}
which proves~\eqref{eq zshift1}. The proof of~\eqref{eq zshift2} is identical.

Iterating these identities gives
\begin{align}
\sum_{n\ge 0}\frac{(-zq^{2k+1};q^2)_n q^n}{(zq^{2k+2};q^2)_n}
&=\frac{1}{q^k}\frac{(zq^2;q^2)_k}{(-zq;q^2)_k}F(z,q)
-\sum_{j=1}^k\frac{1}{q^{k-j+1}}
\frac{(zq^{2j};q^2)_{k-j+1}}{(-zq^{2j-1};q^2)_{k-j+1}}, \label{eq ziterate1}\\
\sum_{n\ge 0}\frac{(-zq^{2k+1};q^2)_n q^n}{(zq^{2k};q^2)_n}
&=\frac{1}{q^k}\frac{(z;q^2)_k}{(-zq;q^2)_k}H(z,q)
-\sum_{j=1}^k\frac{1}{q^{k-j+1}}
\frac{(zq^{2j-2};q^2)_{k-j+1}}{(-zq^{2j-1};q^2)_{k-j+1}}. \label{eq ziterate2}
\end{align}

We now find closed forms for $\mathcal A(z,q)$ and $\mathcal B(z,q)$. Starting from~\eqref{eq zA-def} and using~\eqref{eq ziterate1}, we obtain
\begin{align*}
\mathcal A(z,q)
&=\sum_{k\ge 1} q^{2k}\sum_{l\ge k+1}
\frac{(-zq^{2k+1};q^2)_{l-k-1}}{(zq^{2k};q^2)_{l-k}}q^{l-1}\\
&=\sum_{k\ge 1}\frac{q^{3k}}{1-zq^{2k}}
\sum_{l\ge 0}\frac{(-zq^{2k+1};q^2)_l q^l}{(zq^{2k+2};q^2)_l}\\
&=F(z,q)\sum_{k\ge 1}\frac{q^{2k}(zq^2;q^2)_{k-1}}{(-zq;q^2)_k}
-\sum_{j\ge 1}\frac{q^{3j-1}}{1+zq^{2j-1}}
\sum_{k\ge 0}\frac{(zq^{2j};q^2)_k q^{2k}}{(-zq^{2j+1};q^2)_k}.
\end{align*}
Applying~\eqref{Andr-Subb-Vidy} with $q\to q^2$, the first sum becomes
\begin{align*}
\sum_{k\ge 1}\frac{q^{2k}(zq^2;q^2)_{k-1}}{(-zq;q^2)_k}
&=\frac{1}{1+zq}\sum_{k\ge 0}\frac{(zq^2;q^2)_k q^{2k+2}}{(-zq^3;q^2)_k}\\
&=-\frac{q(zq^2;q^2)_\infty}{z(1+q)(-zq;q^2)_\infty}+\frac{q}{z(1+q)}.
\end{align*}
Similarly,
\begin{align*}
&\sum_{j\ge 1}\frac{q^{3j-1}}{1+zq^{2j-1}}
\sum_{k\ge 0}\frac{(zq^{2j};q^2)_k q^{2k}}{(-zq^{2j+1};q^2)_k} \\
&=\sum_{j\ge 1}\frac{q^{3j-1}}{1+zq^{2j-1}}
\left(
-\frac{(zq^{2j};q^2)_\infty}{zq^{2j-1}(1+q)(-zq^{2j+1};q^2)_\infty}
+\frac{1+q^{1-2j}/z}{1+q}
\right)\\
&=-\frac{1}{z(1+q)}\sum_{j\ge 1}q^j
\frac{(zq^{2j};q^2)_\infty}{(-zq^{2j-1};q^2)_\infty}
+\frac{q}{z(1-q^2)}\\
&=-\frac{q(zq^2;q^2)_\infty}{z(1+q)(-zq;q^2)_\infty}F(z,q)
+\frac{q}{z(1-q^2)}.
\end{align*}
Therefore
\[
z\mathcal A(z,q)=\frac{q}{1+q}F(z,q)-\frac{q}{1-q^2},
\]
which is~\eqref{eq zA-closed}.

The same method applied to~\eqref{eq zB-def} and~\eqref{eq ziterate2} gives
\begin{align*}
\mathcal B(z,q)
&=\sum_{k\ge 1}q^{2k}\sum_{l\ge k+1}
\frac{(-zq^{2k+1};q^2)_{l-k}}{(zq^{2k};q^2)_{l-k}}q^{l-1}\\
&=\sum_{k\ge 1}q^{3k-1}\sum_{l\ge 0}
\frac{(-zq^{2k+1};q^2)_l q^l}{(zq^{2k};q^2)_l}-\frac{q^2}{1-q^3}\\
&=\frac{1}{q}H(z,q)\sum_{k\ge 1}q^{2k}\frac{(z;q^2)_k}{(-zq;q^2)_k}
-D(z,q)-\frac{q^2}{1-q^3},
\end{align*}
where
\[
D(z,q):=\sum_{k\ge 1}q^{3k-1}\sum_{j=1}^k\frac{1}{q^{k-j+1}}
\frac{(zq^{2j-2};q^2)_{k-j+1}}{(-zq^{2j-1};q^2)_{k-j+1}}.
\]
Again using~\eqref{Andr-Subb-Vidy} with $q\to q^2$, we obtain
\begin{align*}
\frac{1}{q}\sum_{k\ge 1}q^{2k}\frac{(z;q^2)_k}{(-zq;q^2)_k}
&=\frac{1-z}{q(1+zq)}\sum_{k\ge 0}\frac{(zq^2;q^2)_k q^{2k+2}}{(-zq^3;q^2)_k}\\
&=-\frac{(z;q^2)_\infty}{z(1+q)(-zq;q^2)_\infty}+\frac{1-z}{z(1+q)},
\end{align*}
and
\begin{align*}
D(z,q)
&=\sum_{j\ge 0}q^{3j-1}\sum_{k\ge 0}q^{2k}\frac{(zq^{2j};q^2)_k}{(-zq^{2j+1};q^2)_k}-\frac{1}{q(1-q^3)}\\
&=\sum_{j\ge 0}q^{3j-1}
\left(
-\frac{(zq^{2j};q^2)_\infty}{zq^{2j-1}(1+q)(-zq^{2j+1};q^2)_\infty}
+\frac{1+q^{1-2j}/z}{1+q}
\right)-\frac{1}{q(1-q^3)}\\
&=-\frac{1}{z(1+q)}\sum_{j\ge 0}q^j
\frac{(zq^{2j};q^2)_\infty}{(-zq^{2j+1};q^2)_\infty}
+\frac{1}{z(1-q^2)}-\frac{1}{(1+q)(1-q^3)}\\
&=-\frac{(z;q^2)_\infty}{z(1+q)(-zq;q^2)_\infty}H(z,q)
+\frac{1}{z(1-q^2)}-\frac{1}{(1+q)(1-q^3)}.
\end{align*}
Thus
\[
z\mathcal B(z,q)=\frac{1-z}{1+q}H(z,q)-\frac{1}{1-q^2}
+\frac{z}{(1+q)(1-q^3)}-\frac{zq^2}{1-q^3},
\]
which is~\eqref{eq zB-closed}.

It remains to derive the master identity. Let
\[
\nu_n(z,q):=\frac{(-zq;q^2)_n q^n}{(zq^2;q^2)_n}.
\]
Then
\[
F(z,q)=\sum_{n\ge 0}\nu_n(z,q).
\]
Also, since
\[
(z;q^2)_n=\frac{1-z}{1-zq^{2n}}(zq^2;q^2)_n,
\]
we have
\[
(1-z)H(z,q)=\sum_{n\ge 0}(1-zq^{2n})\nu_n(z,q).
\]
Furthermore,
\[
(1-zq^{2n+2})\nu_{n+1}(z,q)=q(1+zq^{2n+1})\nu_n(z,q),
\]
so
\begin{align*}
&(1-zq^{2n})\nu_n(z,q)-(1-zq^{2n+2})\nu_{n+1}(z,q)\\
&=\big((1-q)-z(1+q^2)q^{2n}\big)\nu_n(z,q)\\
&=\big((1+q^2)(1-zq^{2n})-q(1+q)\big)\nu_n(z,q).
\end{align*}
Summing over $n\ge 0$ yields the telescoping identity
\begin{equation}\label{eq FH-relation}
(1+q^2)(1-z)H(z,q)-q(1+q)F(z,q)=1-z.
\end{equation}
Now multiplying the desired identity by $z$ and using~\eqref{eq zA-closed} and~\eqref{eq zB-closed}, we find
\begin{align*}
z\big((1+q^2)\mathcal B(z,q)-(1+q)\mathcal A(z,q)\big)
&=\frac{1+q^2}{1+q}(1-z)H(z,q)-qF(z,q)-\frac{1}{1+q} \\
&\quad +z\left(\frac{1+q^2}{(1+q)(1-q^3)}-\frac{q^2(1+q^2)}{1-q^3}\right).
\end{align*}
By~\eqref{eq FH-relation}, the first two terms equal $(1-z)/(1+q)$, and hence
\begin{align*}
z\big((1+q^2)\mathcal B(z,q)-(1+q)\mathcal A(z,q)\big)
&=-\frac{z}{1+q}+z\left(\frac{1+q^2}{(1+q)(1-q^3)}-\frac{q^2(1+q^2)}{1-q^3}\right)\\
&=-\frac{zq^4}{1-q^3}.
\end{align*}
Cancelling the common factor $z$ proves~\eqref{eq master bivariate} and completes the proof of the Main Theorem.
\end{proof}

\section{A Combinatorial Proof For the Main Theorem}\label{combinatorialproofs}

In this section, we give bijective proofs for \eqref{eq zA-closed} and \eqref{eq zB-closed}. We also give an involutive proof of \eqref{eq master bivariate}. 

\subsection{A Bijective Proof for \eqref{eq zA-closed}} Recall \eqref{eq zA-closed} states that
\begin{equation*}
z\mathcal{A}(z;q)=\frac{q}{1+q}F(z,q)-\frac{q}{1-q^2},
\end{equation*}
which is equivalent to
\begin{equation}\label{eq zA-closed Rewrite}
(1+q)z\mathcal{A}(z,q)=qF(z,q)-\frac{q}{1-q}.
\end{equation}
The left hand side is clearly the bivariate generating function for pairs of partitions $(\lambda,\mu)$, where $\lambda$ is either $(1)$ or the empty partition, $\mu$ is a partition counted by $a_2(n;r)$, $q$ keeps track of the total weight of $\lambda$ and $\mu$, and $z$ keeps track of the number of parts greater than $1$ in $\mu$. So, we first give a partition interpretation for the right-hand side.

Let $f(n;r)$ be the number of pairs of partitions $(\lambda,\mu)$ such that $\lambda=(1^{l})$ for some $l>0$, $\mu$ is a partition with parts bounded by $2l-2$ and odd parts are distinct, the total weight of $\lambda$ and $\mu$ is $n$, and the length of $\mu$ is $r$. It is clear that
$$\sum_{n\geq0}\sum_{r\geq0}f(n;r)z^rq^n=\sum_{l\geq0}\frac{(-zq;q^2)_l}{(zq^2;q^2)_l}q^{l+1}=qF(z,q).$$
The term $1/(1-q)$ is the case when $\mu$ is empty. So, the right-hand side of \eqref{eq zA-closed Rewrite} is the generating function for those pairs counted by $f(n;r)$ with $\mu\neq\emptyset$. Now, we are ready to give the bijection.
\begin{proof}[A bijective proof for \eqref{eq zA-closed Rewrite}] Let $(\lambda,\mu)$ be a pair of partitions counted by $f(n;r)$ with $\mu\neq\emptyset$. We consider the following two cases.

\textbf{Case 1.} The smallest part of $\mu$ is even. Note that since $\mu$ is non-empty, the smallest part must be $2k$ with $k>0$. Now, we keep only one copy of $1$ in $\lambda$ and move all the other $1$'s to $\mu$. Let $(\lambda',\mu')$ be the resulting pair. Suppose the length of $\lambda$ is $l$. Note that the parts in $\mu$ are in the interval $[2k,2l-2]$. Now the number of $1$'s in $\mu'$ is $l-1$, so $\mu'$ is counted by $a_2(n;r)$. In this case, $\lambda'=(1)$.

\textbf{Case 2.} The smallest part of $\mu$ is odd. Suppose it is $2k-1$ with $k>0$, since the odd parts in $\mu$ are distinct, we can add one of those $1$'s from $\lambda$ to the smallest part of $\mu$ and move the other $1$'s to $\mu$. Let the resulting pair be $(\lambda',\mu')$. Now, assume the length of $\lambda$ is $l$, then the parts in $\mu$ are in the interval $[2k-1,2l-2]$. So, the smallest even part in $\mu'$ will be $2k$, the number of $1$'s in $\mu'$ will be $l-1$, and the other parts in $\mu'$ are in the interval $[2k,2l-2]$. Thus, $\mu'$ is counted by $a_2(n;r)$, and in this case $\lambda'$ is empty.

By $\lambda'$, we can easily distinguish these two cases, and this map can be easily inverted. So, we finish the proof.
\end{proof}

\subsection{A Bijective Proof for \eqref{eq zB-closed}} Recall \eqref{eq zB-closed} states that
\begin{equation*}
z\mathcal B(z,q)=\frac{1-z}{1+q}H(z,q)-\frac{1}{1-q^2}
+\frac{z}{(1+q)(1-q^3)}-\frac{zq^2}{1-q^3},   
\end{equation*}
which is equivalent to
\begin{equation*}
(1+q)z\mathcal{B}(z,q)=(1-z)H(z,q)-\frac{1}{1-q}+\frac{z}{1-q^3}-\frac{zq^2(1+q)}{1-q^3}.
\end{equation*}
Now, for the right-hand side,
\begin{align*}
&(1-z)H(z,q)-\frac{1}{1-q}+\frac{z}{1-q^3}-\frac{zq^2(1+q)}{1-q^3}\\
=&(1-z)\left(1+\sum_{n\geq1}\frac{(-zq;q^2)_n}{(z;q^2)_n}q^n\right)-\frac{1}{1-q}+\frac{z-zq^2-zq^3}{1-q^3}\\
=&1-z+\sum_{n\geq1}\frac{(-zq;q^2)_n}{(zq^2;q^2)_{n-1}}q^n-\frac{1}{1-q}+z-\frac{zq^2}{1-q^3}\\
=&1+\sum_{n\geq1}\frac{(-zq;q^2)_n}{(zq^2;q^2)_{n-1}}q^n-\frac{1}{1-q}-\frac{zq^2}{1-q^3}.
\end{align*}
So, \eqref{eq zB-closed} can be rewritten as
\begin{equation}\label{eq zB Closed Rewrite}
(1+q)z\mathcal{B}(z,q)=1+\sum_{n\geq1}\frac{(-zq;q^2)_n}{(zq^2;q^2)_{n-1}}q^n-\frac{1}{1-q}-\frac{zq^2}{1-q^3},   
\end{equation}
and we shall give a bijective proof based on this form. The left-hand side is the generating function for pairs of partitions $(\lambda,\mu)$ such that $\lambda$ is either $(1)$ or empty, $\mu$ is a partition counted by $b_2(n;r)$, $q$ keeps track of the total weight of $\lambda$ and $\mu$, and $z$ keeps track of the number of parts of $\mu$ greater than $1$.

Next, we give an interpretation for the right-hand side of \eqref{eq zB Closed Rewrite}. Let $h(n;r)$ be the number of pairs of partitions $(\lambda,\mu)$ such that $\lambda=(1^l)$ for some $l\geq0$, $\mu$ is a partition with parts bounded by $2l-1$ and odd parts are distinct, the total weight of $\lambda$ and $\mu$ is $n$, and the length of $\mu$ is $r$. Then,
$$\sum_{n\geq0}\sum_{r\geq0}h(n;r)z^rq^{n}=1+\sum_{l\geq1}\frac{(-zq;q^2)_l}{(zq^2;q^2)_{l-1}}q^{l}.$$
Note that $1/(1-q)$ is generating all such pairs with $\mu=\emptyset$, while $zq^2/(1-q^3)$ generates all the pairs $((1^l),(2l-1))$ for $l\geq1$. So, the right-hand side of \eqref{eq zB Closed Rewrite} is the generating function for all pairs counted by $h(n;r)$ with $\mu\neq\emptyset$ and $(\lambda,\mu)\neq((1^l),(2l-1))$. Now, we present the bijection.

\begin{proof}[A bijective proof for \eqref{eq zB Closed Rewrite}] Let $(\lambda,\mu)$ be a pair of partitions counted by $h(n;r)$ and satisfy the extra restrictions of the right-hand side. Since, $\mu\neq\emptyset$, this suggests that $\lambda\neq\emptyset$ as well. Next, we consider the following two cases.

\textbf{Case 1.} The smallest part of $\mu$ is even. Assume it is $2k$ with $k>0$, and suppose the length of $\lambda$ is $l$, then the parts of $\mu$ lie in the interval $[2k,2l-1]$. Now, we move $l-1$ ones from $\lambda$ to $\mu$, and let the resulting pair be $(\lambda',\mu')$. Then, the number of $1$'s in $\mu'$ is $l-1$, and other parts in $\mu'$ are in the interval $[2k,2l-1]$. Also, in $\mu'$, the odd parts greater than $1$ are distinct. So, $\mu'$ is counted by $b_2(n;r)$, and $\lambda'=(1)$ in this case.

\textbf{Case 2.} The smallest part of $\mu$ is odd. Assume it is $2k-1$, and suppose the length of $\lambda$ is $l$, then $k$ must be strictly less than $l$. Otherwise, the pair we have would be $((1^l),(2l-1))$, which is excluded form the right-hand side. Now, we add one of the $1$'s in $\lambda$ to the smallest part of $\mu$, then move the rest $1$'s to $\mu$. Let the resulting pair be $(\lambda',\mu')$. Then, the smallest part of $\mu'$ is $2k$, which is smaller than $2l-1$. So, parts greater than $1$ in $\mu'$ are in the interval $[2k,2l-1]$, while the number of $1$'s in $\mu'$ is $l-1$. Thus, $\mu'$ is counted by $b_2(n;r)$ and $\lambda'=\emptyset$ in this case.

By $\lambda'$ we can distinguish the two cases, then the inverse of this map is straightforward. So, we finish the proof. 
\end{proof}

\subsection{Involutive proof of \eqref{eq master bivariate}} 
We prove \eqref{eq master bivariate} using a sign-reversing involution. The point of the proof is that the two partition families counted by
\(\mathcal A(z,q)\) and \(\mathcal B(z,q)\) are almost the same.  The only
difference is that the \(\mathcal{B}\)-family allows one additional largest odd part than the $\mathcal{A}$-family.
After the two sides are modified by the factors \(1+q^2\) and \(1+q\), this
extra freedom is cancelled by a sign-reversing involution.  The only objects
which cannot be paired form a simple boundary family, and this family gives
the correction term on the right-hand side of (1.6).

We first recall the combinatorial meaning of the two generating functions.
An object counted by \(\mathcal A(z,q)\) and $\mathcal B(z,q)$ is determined by integers $l\ge2, 1\le k\le l-1$, together with a choice of optional parts. They have the following structure:
\begin{itemize}
    \item contains the compulsory part \(2k\);
    \item contains \(l-1\) compulsory parts equal to \(1\);
    \item may contain optional even parts from $2k,\,2k+2,\,2k+4,\ldots,2l-2$, with repetitions allowed;
    \item may contain optional odd parts from $2k+1,\,2k+3,\ldots,2l-3+2i$, each at most once, where $i$ is $0$ or $1$, according as partitions are enumerated by $\mathcal{A}(z,q)$ or $\mathcal{B}(z,q)$, respectively.
\end{itemize} 
The exponent of \(q\) records the size of the partition, while the exponent
of \(z\) records the number of optional parts greater than \(1\).  Thus every
optional part contributes one factor of \(z\), in addition to its usual weight, accounted for by the exponent of $q$. It is convenient to put $m=l-k$. Then \(m\ge1\), and the largest optional even part for both sets is
\[
E_m=2l-2=2k+2m-2,
\]
while the largest odd part available only in the \(\mathcal{B}\)-family is
\[
O_m=2l-1=2k+2m-1.
\]
Thus \(O_m=E_m+1\). We now interpret the left-hand side of (1.6) as a signed generating function.  The positive objects are the \(\mathcal B\)-objects,
together with an auxiliary marker of weight either \(1\) or \(q^2\). These
are counted by $(1+q^2)\mathcal B(z,q)$. The negative objects are the \(\mathcal A\)-objects, together with an auxiliary marker
of weight either \(1\) or \(q\). These are counted by $(1+q)\mathcal A(z,q)$, and are assigned negative signs.

We shall define a weight-preserving, sign-reversing involution on all augmented
objects except for one boundary family. Here is a flow-chart on how the involution pairs off partitions on both the sets:
\[
\begin{array}{rclcl}
\overline{\mathcal B}(1,\;O_m\notin\lambda)
&\longleftrightarrow&
\overline{\mathcal A}(1)
&\qquad&
\text{common cancellation},\\[2mm]

\overline{\mathcal A}(q,\;E_m\in\lambda)
&\longleftrightarrow&
\overline{\mathcal B}(1,\;O_m\in\lambda)
&\qquad&
q+E_m\leftrightarrow O_m,\\[2mm]

\overline{\mathcal A}(q,\;E_m\notin\lambda,\;m\ge2)
&\longleftrightarrow&
\overline{\mathcal B}(q^2,\;m-1)
&\qquad&
m\mapsto m-1,\\[2mm]

\overline{\mathcal A}(q,\;E_1\notin\lambda,\;m=1)
&\longmapsto&
\text{unpaired}
&\qquad&
(2k,1^k)\text{ with marker }q
\end{array}
\]
where we used the notations $\overline{\mathcal A}(\text{marker},\text{conditions})$ and $\overline{\mathcal B}(\text{marker},\text{conditions})$ to represent a partition enumerated by $(1+q)\mathcal A(z,q)$ and $(1+q^2)\mathcal B(z,q)$ with the markers and conditions on the parts. We omit the conditions from the notation when there are no conditions on the parts. We now describe the involution in detail. 

First, there is an immediate cancellation.  If a positive \(\mathcal B\)-object carries the marker of weight \(1\) and does not contain the extra top odd part \(O_m=2l-1\), then it is already
an \(\mathcal A\)-object with the same parameters and the same optional parts.  We pair
it with the corresponding negative \(\mathcal A\)-object carrying the marker of weight
\(1\).  This pairing is clearly sign-reversing and weight-preserving. It remains to pair three types of objects:

\[
\begin{array}{ll}
\text{negative:}&\mathcal A\text{-objects with marker }q\\
\text{positive:} & \mathcal B\text{-objects with marker }q^2,\\
\text{positive:} & \mathcal B\text{-objects with marker }1\text{ containing }O_m.
\end{array}
\]
Take a negative \(\mathcal A\)-object with marker \(q\), and fix its parameters \(k\) and \(m\).  There are two possibilities.

\noindent \textbf{Case 1} ($E_m$ is in $\mathcal A$): Suppose first that the object contains at least one copy of the largest
optional even part $E_m$. Then we remove the marker \(q\) and remove one copy of \(E_m\).  In their place we insert the top odd part $O_m$. The resulting object is a positive \(\mathcal B\)-object with marker \(1\). The weight is unchanged, because
\[
q\cdot zq^{E_m}=zq^{E_m+1}=zq^{O_m}.
\]
The operation is reversible: starting with a positive \(\mathcal B\)-object with marker
\(1\) which contains \(O_m\), we remove \(O_m\), insert one copy of \(E_m\), and restore the negative marker \(q\).  Thus these two classes cancel
perfectly.

\noindent \textbf{Case 2} ($E_m$ is not in $\mathcal A$): Now suppose that the negative \(\mathcal A\)-object with marker \(q\) contains no copy
of its largest optional even part \(E_m\).  If \(m\ge2\), then the same
optional parts are still legal after lowering the interval length from \(m\)
to \(m-1\).  Indeed, the even parts now lie among
\[
2k,2k+2,\ldots,2k+2m-4,
\]
and the odd parts lie among
\[
2k+1,2k+3,\ldots,2k+2m-3.
\]
But these are precisely the optional even and odd parts allowed for a
\(\mathcal B\)-object with parameters \(k\) and \(m-1\). We need to check if the weight of the partition coming from the compulsory parts changes. The compulsory contribution for the $\mathcal A$ object is $q^{2k+l-1}=q^{3k+m-1}$, and the marker contributes an additional factor \(q\).  Hence the compulsory marker contribution is $q^{3k+m}$. After lowering \(m\) to \(m-1\), the compulsory contribution becomes $q^{3k+(m-1)-1}=q^{3k+m-2}$, and the $q^2$ marker in $\mathcal B$ contributes \(q^2\).  Thus the new compulsory-marker contribution is again $q^{3k+m-2}q^2=q^{3k+m}$. Since the optional parts were not changed, the total weight is unchanged. This operation is also sign-reversing and reversible. 
Indeed, starting with a positive \(\mathcal B\)-object
with marker \(q^2\) and parameters \(k, m-1\) ($m\ge 2$), we raise the interval length to
\(m\), keep the optional parts unchanged, and replace the marker \(q^2\) by
the negative marker \(q\).  The resulting \(\mathcal A\)-object has no copy of the new
largest optional even part \(E_m\), exactly as required.

We have therefore paired all objects except those for which the second
operation cannot be performed.  This happens precisely when \(m=1\).  In that
case the largest optional even part is \(E_1=2k\), and if the negative \(\mathcal A\)-object contains no copy of \(E_1\), then it contains no optional parts at all (since the smallest odd part for $\mathcal A$ is $2k+1$).  Since \(m=1\), we have \(l=k+1\).  The unpaired $\mathcal{A}$ object with negative marker $q$ is therefore
the partition
\[
(2k,\underbrace{1,1,\cdots,1}_{k-\text{ones}}),
\]
Its weight is $q^{2k+k}q=q^{3k+1}$. Consequently the total contribution of the unpaired objects is
\[
-\sum_{k\ge1}q^{3k+1}
=
-\frac{q^4}{1-q^3}.
\]
All the remaining objects cancel in pairs under the weight-preserving
sign-reversing involution.  Hence the signed generating function is exactly
the contribution of the unpaired boundary family, and we obtain
\[
(1+q^2)\mathcal B(z,q)-(1+q)\mathcal A(z,q)
=
-\frac{q^4}{1-q^3}.
\]
This proves \eqref{eq master bivariate}.\qed

\section{Proofs of Theorems~\ref{thm main-1} and~\ref{thm main-2}}\label{sec proofs main theorems}
We first prove Theorem~\ref{thm main-1} from the bivariate formula.  We then give an independent direct proof based on Heine's transformation and the Rogers--Fine identity.  Finally, we prove Theorem~\ref{thm main-2}.

\subsection{Proof of Theorem~\ref{thm main-1} from the master identity}\label{sec proof main-1}
Set $z=-1$ in~\eqref{eq zA-closed}. Since
\[
F(-1,q)=\sum_{n\geq 0}\fr{(q;q^2)_n q^n}{(-q^2;q^2)_n},
\]
we obtain
\begin{equation}\label{eq A-from-F}
A(q)=\fr{q}{1-q^2}-\fr{q}{1+q}F(-1,q).
\end{equation}
Thus it remains only to evaluate this single basic hypergeometric series.

Recall the Rogers--Fine identity~\cite[p.223, Eq. (9.1.1)]{Andrews-Berndt 1}
\begin{equation}\label{Rog-Fin}
\sum_{n\geq 0}\fr{(\alpha;q)_n}{(\beta;q)_n}\tau^n
=\sum_{n\geq 0}\fr{(\alpha;q)_n (\alpha\tau q/\beta;q)_n (1-\alpha\tau q^{2n})\beta^n \tau^n q^{n^2-n}}{(\beta;q)_n (\tau ;q)_{n+1}}.
\end{equation}
Applying~\eqref{Rog-Fin} with $q \to q^2$ and $(\alpha,\beta,\tau)=(q,-q^2,q)$, then simplifying, gives
\begin{align*}
F(-1,q)
&=\sum_{n\geq 0}\fr{(q;q^2)_n q^n}{(-q^2;q^2)_n}\\
&=\sum_{n\geq 0}\fr{(q;q^2)_n (-q^2;q^2)_n (1-q^{4n+2})(-1)^n q^{2n^2+n}}{(-q^2;q^2)_n (q;q^2)_{n+1}} \\
&=\sum_{n\geq 0}(-1)^n q^{2n^2+n}(1+q^{2n+1}) \\
&=\sum_{n\geq 0}\Big( q^{8n^2+2n}(1+q^{4n+1})-q^{8n^2+10n+3}(1+q^{4n+3}) \Big) \\
&= \sum_{n\geq 0} \Big( q^{\binom{4n+1}{2}}+q^{\binom{4n+2}{2}}-q^{\binom{4n+3}{2}}-q^{\binom{4n+4}{2}} \Big).
\end{align*}
Hence
\begin{equation}\label{eq a-false-theta}
F(-1,q)=-\sum_{n\geq 0} \chi(n) q^{\binom{n+1}{2}}.
\end{equation}
Substituting~\eqref{eq a-false-theta} into~\eqref{eq A-from-F} yields
\[
A(q)=\fr{q}{1-q^2}+\fr{1}{1+q}\sum_{n\geq 0}\chi(n)q^{\binom{n+1}{2}+1},
\]
which is exactly~\eqref{conj1-id}. This proves Theorem~\ref{thm main-1}.

\subsection{A direct Heine--Rogers--Fine proof of Theorem~\ref{thm main-1}}\label{sec direct proof main-1}
For completeness, and also to record the complementary method, we give a direct proof of Theorem~\ref{thm main-1} which does not use the bivariate master identity.  Starting from the defining double series and writing $r=l-k$, we have
\begin{equation}\label{eq direct A start}
A(q)=\sum_{k,r\ge 1}\frac{(q^{2k+1};q^2)_{r-1}}{(-q^{2k};q^2)_r}q^{3k+r-1}.
\end{equation}
Putting $j=r-1$ gives
\begin{equation}\label{eq direct A Ik}
A(q)=\sum_{k\ge1}\frac{q^{3k}}{1+q^{2k}}I_k(q),\qquad
I_k(q):=\sum_{j\ge0}\frac{(q^{2k+1};q^2)_j}{(-q^{2k+2};q^2)_j}q^j.
\end{equation}
We regard $I_k(q)$ as
\[
{}_{2}\phi_{1}\left(\begin{matrix}q^{2k+1},q^2\\-q^{2k+2}\end{matrix};q^2,q\right).
\]
We use Heine's transformation in the form
\[
{}_{2}\phi_{1}\left(\begin{matrix}a,b\\c\end{matrix};Q,z\right)
=\frac{(abz/c;Q)_\infty}{(z;Q)_\infty}
{}_{2}\phi_{1}\left(\begin{matrix}c/a,c/b\\c\end{matrix};Q,abz/c\right),
\]
where
\[
{}_{2}\phi_{1}\left(\begin{matrix}a,b\\c\end{matrix};Q,z\right)
:=\sum_{n\ge0}\frac{(a,b;Q)_n}{(c,Q;Q)_n}z^n.
\]
With base $q^2$ and parameters
\[
a=q^{2k+1},\qquad b=q^2,\qquad c=-q^{2k+2},\qquad z=q,
\]
yields
\begin{align}
I_k(q)
&=\frac{(-q^2;q^2)_\infty}{(q;q^2)_\infty}
\sum_{n\ge0}\frac{(-q;q^2)_n(-q^{2k};q^2)_n}{(q^2;q^2)_n(-q^{2k+2};q^2)_n}(-q^2)^n \notag\\
&=\frac{(-q^2;q^2)_\infty}{(q;q^2)_\infty}
\sum_{n\ge0}\frac{1+q^{2k}}{1+q^{2k+2n}}\frac{(-q;q^2)_n}{(q^2;q^2)_n}(-q^2)^n . \label{eq direct Ik heine}
\end{align}
Substituting~\eqref{eq direct Ik heine} into~\eqref{eq direct A Ik}, expanding
$(1+q^{2k+2n})^{-1}$ as a geometric series, and interchanging sums gives
\begin{align}
A(q)
&=\frac{(-q^2;q^2)_\infty}{(q;q^2)_\infty}
\sum_{h\ge0}(-1)^h\frac{q^{2h+3}}{1-q^{2h+3}}
\sum_{n\ge0}\frac{(-q;q^2)_n}{(q^2;q^2)_n}(-q^{2h+2})^n. \label{eq direct geom}
\end{align}
The $q$-binomial theorem evaluates the inner sum as
\[
\sum_{n\ge0}\frac{(-q;q^2)_n}{(q^2;q^2)_n}(-q^{2h+2})^n
=\frac{(q^{2h+3};q^2)_\infty}{(-q^{2h+2};q^2)_\infty}.
\]
Hence
\begin{align}
A(q)
&=\sum_{h\ge0}(-1)^h
\frac{q^{2h+3}(-q^2;q^2)_h}{(1-q^{2h+3})(q;q^2)_{h+1}} \notag\\
&=\frac{q^3}{(1-q)(1-q^3)}
\sum_{h\ge0}\frac{(-q^2;q^2)_h(-q^2)^h}{(q^5;q^2)_h}. \label{eq direct before RF}
\end{align}
Applying the Rogers--Fine identity~\eqref{Rog-Fin} with base $q^2$ and
\[
\alpha=-q^2,\qquad \beta=q^5,\qquad \tau=-q^2,
\]
we obtain from~\eqref{eq direct before RF}
\begin{equation}\label{eq direct after RF}
A(q)=q^3\sum_{n\ge0}(-1)^n q^{2n^2+5n}
\frac{1-q^{2n+2}}{(1-q^{2n+1})(1-q^{2n+3})}.
\end{equation}
Using
\[
\frac{1-q^{2n+2}}{(1-q^{2n+1})(1-q^{2n+3})}
=\frac{1}{(1+q)(1-q^{2n+1})}+\frac{q}{(1+q)(1-q^{2n+3})},
\]
then shifting the second sum in the resulting expression by $n\mapsto n-1$, gives
\begin{align*}
A(q)
&=\frac{1}{1+q}\left(\frac{q^3}{1-q}
-\sum_{n\ge1}(-1)^nq^{2n^2+n+1}(1+q^{2n+1})\right)\\
&=\frac{q}{1-q^2}-\frac{q}{1+q}
\sum_{n\ge0}(-1)^nq^{2n^2+n}(1+q^{2n+1}).
\end{align*}
Finally,
\[
\sum_{s\ge0}\chi(s)q^{\binom{s+1}{2}}
=-\sum_{n\ge0}(-1)^nq^{2n^2+n}(1+q^{2n+1}),
\]
which again gives
\[
A(q)=\frac{q}{1-q^2}+\frac{1}{1+q}\sum_{s\ge0}\chi(s)q^{\binom{s+1}{2}+1}.
\]
This is Theorem~\ref{thm main-1}.

\subsection{Proof of Theorem~\ref{thm main-2}}\label{sec proof main-2}
The proof separates the derivation of the generating function from the coefficient estimate needed for the range assertion.  We first isolate the false theta contribution.

\begin{lemma}\label{lem false theta block}
Let
\[
C(q):=\frac{q}{1+q^2}\sum_{n\ge 0}\chi(n)q^{\binom{n+1}{2}}=\sum_{n\ge 0} c(n)q^n.
\]
Also let $T_r:=\binom{r+1}{2}$ for $r\ge 0$.  Fix $m\ge 0$ and $u\in\{1,2,\dots,8\}$.  If
\[
T_{8m+u-1}<n\le T_{8m+u},
\]
then
\begin{equation}\label{eq exact c(n)}
c(n)=\mu_u(n)\,\chi(n-1),
\end{equation}
where $\mu_u(n)$ is defined in~\eqref{eq mu-definition}. In particular, $c(n)=0$ whenever $T_{8m+7}<n\le T_{8m+8}$.
\end{lemma}

\begin{proof}
For $m\ge 0$, set
\[
t(m):=\frac{1}{1+q^2}\sum_{j=8m}^{8m+7}\chi(j)q^{\binom{j+1}{2}+1}.
\]
Then $C(q)=\sum_{m\ge 0} t(m)$.  Pairing the terms with indices $j$ and $j+4$ gives
\begin{equation}\label{id lem main-2}
t(m)=\sum_{j=8m}^{8m+3}\chi(j)q^{\binom{j+1}{2}+1}(1-q^2+q^4-\cdots+q^{4j+8}).
\end{equation}
Indeed, since $\chi(j+4)=\chi(j)$ and
\[
\binom{j+5}{2}-\binom{j+1}{2}=4j+10,
\]
we have
\[
\frac{q^{\binom{j+1}{2}+1}+q^{\binom{j+5}{2}+1}}{1+q^2}
=q^{\binom{j+1}{2}+1}(1-q^2+q^4-\cdots+q^{4j+8}).
\]
Thus
\begin{align*}
t(m)
&= -q^{T_{8m}+1}(1-q^2+q^4-\cdots + q^{32m+8}) \\
&\quad -q^{T_{8m+1}+1}(1-q^2+q^4-\cdots + q^{32m+12}) \\
&\quad +q^{T_{8m+2}+1}(1-q^2+q^4-\cdots + q^{32m+16}) \\
&\quad +q^{T_{8m+3}+1}(1-q^2+q^4-\cdots + q^{32m+20}).
\end{align*}
The lowest power in $t(m)$ is $T_{8m}+1$ and the highest power is $T_{8m+7}-1=\binom{8m+8}{2}-1$, so the supports of the polynomials $t(m)$ are pairwise disjoint. Hence $c(n)=[q^n]t(m)$ whenever $T_{8m}<n\le T_{8m+8}$.

Now assume $T_{8m+u-1}<n\le T_{8m+u}$.  The summand with index $s\in\{0,1,2,3\}$ has sign $\chi(8m+s)$ and contributes only to exponents
\[
T_{8m+s}<n<T_{8m+s+4},\qquad
n\equiv T_{8m+s}+1\pmod 2.
\]
Thus, before imposing the parity condition, the possible indices $s$ are
\[
s\in
\begin{cases}
\{0\}, & u=1,\\
\{0,1\}, & u=2,\\
\{0,1,2\}, & u=3,\\
\{0,1,2,3\}, & u=4,\\
\{1,2,3\}, & u=5,\\
\{2,3\}, & u=6,\\
\{3\}, & u=7,\\
\varnothing, & u=8.
\end{cases}
\]
At the upper endpoints for the lines $u=5,6,7$, the extra boundary index listed above has the wrong parity and therefore contributes zero; thus the displayed list is harmless once the parity condition is imposed. Since $\chi(8m)=\chi(8m+1)=-1$ and $\chi(8m+2)=\chi(8m+3)=1$, each actual contribution has sign $\chi(n-1)$: the summands with $s=0,3$ contribute when $n$ is odd, while those with $s=1,2$ contribute when $n$ is even. Therefore $c(n)$ equals $\chi(n-1)$ multiplied by the number of possible indices whose parity matches $n$. Counting these indices gives exactly the factor $\mu_u(n)$ in~\eqref{eq exact c(n)}.
\end{proof}

\begin{proof}[Proof of Theorem~\ref{thm main-2}]
First we prove the generating-function formula. By~\eqref{eq weighted relation} and Theorem~\ref{thm main-1},
\begin{align*}
(1+q^2)B(q)
&=(1+q)A(q)-\frac{q^4}{1-q^3}\\
&=(1+q)\left(\frac{q}{1-q^2}+\frac{1}{1+q}\sum_{n\ge 0}\chi(n)q^{\binom{n+1}{2}+1}\right)-\frac{q^4}{1-q^3}\\
&=\frac{q}{1-q}+q\sum_{n\ge 0}\chi(n)q^{\binom{n+1}{2}}-\frac{q^4}{1-q^3}.
\end{align*}
Dividing by $1+q^2$ and simplifying the rational part,
\[
\frac{q}{(1-q)(1+q^2)}-\frac{q^4}{(1-q^3)(1+q^2)}
=\frac{q+q^2+q^3-q^4}{(1-q)(1+q^2)(1+q+q^2)},
\]
gives~\eqref{conj2-id}.

It remains to prove the range assertion.  Let
\[
R(q):= \frac{q+q^2+q^3-q^4}{(1-q)(1+q^2)(1+q+q^2)}.
\]
Then
\begin{align*}
R(q)
&=\frac{q+q^2-q^4+q^5+2q^6-q^7-q^8+2q^9+q^{10}-q^{11}}{1-q^{12}},
\end{align*}
so
\begin{equation}\label{coeff-R}
[q^N] R(q)
\in\begin{cases}
\{0,-1\} &\ \text{if\ } N\equiv 0,3\Mod{4}, \\
\{1,2\} & \ \text{if\ } N\equiv 1,2\Mod{4}.
\end{cases}
\end{equation}
On the other hand, Lemma~\ref{lem false theta block} implies
\begin{equation}\label{coeff-C}
[q^N] C(q)
\in\begin{cases}
\{0,1,2\} &\ \text{if\ } N\equiv 0,3\Mod{4}, \\
\{0,-1,-2\} &\ \text{if\ } N\equiv 1,2\Mod{4}.
\end{cases}
\end{equation}
Since $B(q)=R(q)+C(q)$, equations~\eqref{coeff-R} and~\eqref{coeff-C} give
\[
[q^N]B(q)\in
\begin{cases}
\{0,-1\}+\{0,1,2\}, & N\equiv 0,3\pmod 4,\\
\{1,2\}+\{0,-1,-2\}, & N\equiv 1,2\pmod 4,
\end{cases}
\]
and both displayed sumsets are contained in $\{-1,0,1,2\}$.  This completes the proof of Theorem~\ref{thm main-2}.
\end{proof}

\section{Further consequences of the master identity}\label{sec proofs corollaries}
We first prove the fixed-refinement specialization.
\begin{proof}[Proof of Corollary~\ref{cor fixed-r}]
Setting $z=0$ in~\eqref{eq zA-def} and~\eqref{eq zB-def} gives
\[
A_0(q)=B_0(q)=\sum_{\substack{l\ge 2\\1\le k\le l-1}} q^{2k+l-1}
=\sum_{k\ge 1}\sum_{m\ge 1} q^{3k+m-1}
=\frac{q^3}{(1-q)(1-q^3)}.
\]
Since
\[
\mathcal A(z,q)=\sum_{r\ge 0}A_r(q)z^r,\qquad \mathcal B(z,q)=\sum_{r\ge 0}B_r(q)z^r,
\]
extracting the coefficient of $z^r$ from~\eqref{eq master bivariate} yields~\eqref{eq fixed-r relation} for every $r\ge 1$, and~\eqref{eq fixed-r coeff relation} follows by comparing coefficients of $q^n$.
\end{proof}

The proof of Theorem~\ref{thm main-2} already determines the false theta contribution in triangular blocks.  Combining Lemma~\ref{lem false theta block} with the rational part proves the exact coefficient formula.

\begin{proof}[Proof of Corollary~\ref{cor exact formula for bpp}]
By Theorem~\ref{thm main-2},
\[
\sum_{n\ge 0} b_2''(n)q^n
=\frac{q+q^2+q^3-q^4}{(1-q)(1+q^2)(1+q+q^2)} + C(q).
\]
Also,
\[
\frac{q+q^2+q^3-q^4}{(1-q)(1+q^2)(1+q+q^2)}
=\frac{q+q^2-q^4+q^5+2q^6-q^7-q^8+2q^9+q^{10}-q^{11}}{1-q^{12}},
\]
so its coefficients are exactly $\rho(n)$.  The asserted formula now follows from Lemma~\ref{lem false theta block}.
\end{proof}

\section{A quantum modular consequence}\label{sec quantum}

The direct false theta form above also gives the quantum modular interpretation of the first signed generating function.  Let $q=e^{2\pi i\tau}$ and define
\[
\Theta(\tau):=\frac{q^{1/8}}{1-q}-q^{-7/8}(1+q)A(q).
\]
By Theorem~\ref{thm main-1}, or equivalently by the last displayed formula in the direct derivation,
\begin{equation}\label{eq theta false theta}
\Theta(\tau)=q^{1/8}F(-1,q),\qquad
F(-1,q)=\sum_{n\ge0}(-1)^nq^{2n^2+n}(1+q^{2n+1}).
\end{equation}
Let $\psi$ be the odd periodic function modulo $8$ defined by
\[
\psi(r)=
\begin{cases}
1, & r\equiv 1,3\pmod 8,\\
-1, & r\equiv 5,7\pmod 8,\\
0, & r\equiv 0,2,4,6\pmod 8.
\end{cases}
\]
Then
\begin{equation}\label{eq psi false theta}
F(-1,q)=\sum_{r\ge1}\psi(r)q^{(r^2-1)/8}.
\end{equation}
Indeed, the substitutions $r=4n+1$ and $r=4n+3$ in the two summands of $F(-1,q)$ give exactly the four nonzero residue classes of $\psi$ modulo $8$.


We now place \(\Theta\) in the framework of partial theta series with periodic
coefficients.  Let
\[
\vartheta_{\psi}(z)
:=
\sum_{r\ge0}\psi(r)e^{2\pi i z r^2/16}.
\]
This is the partial theta series \(\theta_f(z)\) of \cite{Goswami-Osburn} with
period \(M=8\) and coefficient function \(f=\psi\).  Since \(\psi\) is odd and
periodic modulo \(8\), their theorem implies that \(\vartheta_{\psi}(z)\) is a
strong quantum modular form (see \cite[pp 453]{Goswami-Osburn}) of weight \(1/2\) on \(\mathbb Q\), with respect to
the group $\Gamma_8=\Gamma_1(16)$. Moreover, $\Theta(\tau)=\vartheta_{\psi}(2\tau)$. Consequently, \(\Theta\) is also a strong quantum modular form of weight
\(1/2\) on \(\mathbb Q\), with respect to the pullback of \(\Gamma_1(16)\) under
the dilation \(\tau\mapsto 2\tau\), where for $N\ge 1$
\begin{equation*}
\Gamma_1(N):=\left\{\begin{pmatrix}
a&b\\c&d    
\end{pmatrix}\in SL_2(\mathbb{Z}): c\equiv 0\!\!\pmod{N},\,a\equiv d\equiv 1\!\!\pmod{N}\right\}.    
\end{equation*}
The associated weight \(3/2\) unary theta series, or shadow, is
\begin{equation}\label{eq:shadow}
\Theta_{\psi}(z)
:=
\sum_{r\in\mathbb Z}r\psi(r)e^{2\pi i z r^2/16}.
\end{equation}
In particular, the quantum modularity of \(\Theta(\tau)\) is governed by the
Eichler integral attached to the unary theta function
\(\Theta_{\psi}(z)\).  More explicitly, 
it is shown in \cite{Goswami-Osburn} that the obstruction function for
\(\vartheta_{\psi}\) is given by an Eichler integral of \(\Theta_{\psi}\). Define the non-holomorphic Eichler integral associated with
\(\Theta_{\psi}\)
\[
\widehat{\Theta}_{\psi}(\tau)
:=
\frac{1}{\sqrt{8i}}
\int_{-\overline{\tau}}^{i\infty}
\Theta_{\psi}(w)(w+\tau)^{-1/2}\,dw
\]
then we have
\begin{equation}\label{qm}
\vartheta_{\psi}(\alpha)
-
\left(\vartheta_{\psi}|_{1/2,\nu}\gamma\right)(\alpha)
=
r_{\gamma,\psi}(\alpha),
\end{equation}
where $\nu$ is a multiplier, $f|_{k,\chi}\gamma$ is the usual slash operator, and 
the modular obstruction is
\begin{equation}\label{eq:quantum-cocycle}
r_{\gamma,\psi}(\tau)
=
\frac{1}{\sqrt{8i}}
\int_{\gamma^{-1}(i\infty)}^{i\infty}
\Theta_{\psi}(w)(w-\tau)^{-1/2}\,dw,\quad \alpha\in\mathbb{Q}.
\end{equation}
The function \(r_{\gamma,\psi}\) is holomorphic on the lower half-plane,
extends smoothly to the real line, and is real-analytic away from the singular
point \(\gamma^{-1}(i\infty)\). 
\begin{theorem}\label{thm:quantum-modularity}
Let $D=\begin{pmatrix}2&0\\0&1\end{pmatrix}$ and set $\Gamma_\Theta=
D^{-1}\Gamma_1(16)D\cap SL_2(\mathbb Z)$. Then \(\Theta(\tau)\) is a strong
quantum modular form of weight \(1/2\) on \(\mathbb Q\) with respect to
\(\Gamma_\Theta\). More precisely, we have 
\begin{equation*}
\Theta(\alpha)-(\Theta|_{1/2,\nu_\Theta}\delta)(\alpha)=R_{\delta,\psi}(\alpha),\quad \delta\in\Gamma_\Theta,\quad\alpha\in\mathbb{Q}    
\end{equation*}
where $\nu_\Theta(\delta)=\nu_\psi(D\delta D^{-1})$ and $\,R_{\delta,\psi}=r_{D\delta D^{-1},\psi}(2\tau)$. The function \(R_{\delta,\psi}\) extends smoothly to the real line and is
real-analytic on $\mathbb R\setminus\{\delta^{-1}(i\infty)\}$.
\end{theorem}
\begin{proof}
Let all notations be as in Theorem \ref{thm:quantum-modularity}. For \(\delta\in\Gamma_{\Theta}\), since \(\Theta(\tau)=\vartheta_{\psi}(2\tau)\) and $\gamma(2\tau)=2\delta\tau$, the slash actions are compatible:
\begin{equation}\label{compatible}
\left(\vartheta_{\psi}\big|_{\frac12,\nu_{\psi}}\gamma\right)(2\tau)
=
\left(\Theta\big|_{\frac12,\nu_{\Theta}}\delta\right)(\tau).    
\end{equation}
Thus \eqref{qm} and \eqref{compatible} gives
\[
\Theta(\alpha)
-
\left(\Theta\big|_{\frac12,\nu_{\Theta}}\delta\right)(\alpha)
=
r_{\gamma,\psi}(2\alpha)=:R_{\delta,\psi}(\alpha).
\]
Since \(r_{\gamma,\psi}\) is real-analytic away from \(\gamma^{-1}(i\infty)\),
\(R_{\delta,\psi}\) is real-analytic away from \(\delta^{-1}(i\infty)\).  This
proves the result.
\end{proof}
\section{Declaration of interests}
The authors declare that they have no known competing financial interests or personal relationships that could have appeared to influence the work reported in this paper.
\end{document}